\documentclass[12pt]{amsart}

\numberwithin{equation}{section}

\newtheorem{The}[subsection]{Theorem}
\newtheorem{Lem}[subsection]{Lemma}

\newtheorem{Cor}[subsection]{Corollary}

\theoremstyle{definition}

\theoremstyle{remark}

\newcommand{\field}{\mathbb{F}}
\newcommand{\bbN}{\mathbb{N}}

\newcommand{\mcB}{\mathcal {B}}
\newcommand{\mcC}{\mathcal {C}}
\newcommand{\mcG}{\mathcal {G}}
\newcommand{\mcV}{\mathcal {V}}
\newcommand{\tZ}{\widetilde{Z}}

\newcommand{\V}{{\mathcal {V}}}

\newcommand{\sltwo}{SL_2(\field_p)}

\newcommand{\slinv}{\field[V]^{G}}

\newcommand{\tr}{\mathop{\rm tr}}

\newcommand{\lt}{\mathop{\rm LT}}
\newcommand{\lm}{\mathop{\rm LM}}
\newcommand{\wt}{\mathop{\rm wt}}

\title[$SL_2(\mathbb{F}_p)$-Invariants of the Third Symmetric Power]
{The Invariants of the Third Symmetric Power Representation of $SL_2(\mathbb{F}_p)$}

\author{Ashley Hobson}
\address{School of Mathematics, Statistics \&  Actuarial Science \\
 \hfil\break\indent University of Kent, Canterbury, CT2 7NF, UK}
\email{ashleyghobson@googlemail.com}

\author{R.\ James Shank}
\address{School of Mathematics, Statistics \&  Actuarial Science \\
 \hfil\break\indent University of Kent, Canterbury, CT2 7NF, UK}
\email{R.J.Shank@kent.ac.uk}

\thanks{The research of the first author was supported by grants from EPSRC}

\date{\today}

\begin{document}
\begin{abstract}
For a prime $p>3$, we compute a finite generating set for the
$SL_2(\mathbb{F}_p)$-invariants of the third symmetric power
representation.  The proof relies on the construction of an infinite
SAGBI basis and uses the Hilbert series calculation of Hughes and
Kemper.
\end{abstract}

\maketitle

\section{Introduction}

Consider the generic binary cubic over a field $\field$ of characteristic not $3$:
\begin{equation*} a_0X^3+3a_1X^2Y+3a_2XY^2+a_3Y^3.\end{equation*}
Identifying
\begin{equation*} 
X=\left[\begin{array}{c} 0\\1\end{array}\right] \,\, \mbox{and} \,\, Y=\left[\begin{array}{c} 1\\0\end{array}\right] 
\end{equation*}
induces a left action of the general linear group $GL_2(\mathbb{F})$ on the third symmetric power
$$V:={\rm Span}_{\field}[\,Y^3,3Y^2X,3YX^2,X^3\,]$$ and a right action on the dual 
$V^*={\rm Span}_{\field}[a_3,a_2,a_1,a_0].$
For example
$$\begin{array}{cccccc}\sigma=\left[ \begin{array}{cc} 1&
1\\ 0&1 \end{array}\right] & {\rm acts\ on}\ V^*\ {\rm as}&
\left[ \begin{array}{cccc} 1&3&3&1\\ 0&1&2&1\\ 0&0&1&1\\ 0&0&0&1 \end{array}\right]&&&
\end{array}
$$
with $a_3=[1\ 0\ 0\ 0]$, $a_2=[0\ 1\ 0\ 0]$, $a_1=[0\ 0\ 1\ 0]$, $a_0=[0\ 0\ 0\ 1]$.
The action on $V^*$ extends to an action by algebra automorphisms on the symmetric algebra 
$\field[V]=\field[a_3,a_2,a_1,a_0]$. For any subgroup $G\leq GL_2(\field)$, we denote the subring of invariant 
polynomials by $\field[V]^G$.

Throughout we assume that $\field$ has characteristic $p>3$. Thus $\field_p\subseteq\field$ and $SL_2(\field_p)\leq GL_2(\field)$.
The primary goal of this paper is to compute a finite generating set for $\field[V]^{SL_2(\field_p)}$.
We note that $V$ is the unique four-dimensional irreducible representation of $\sltwo$ (see, for example, \cite[pp.\,14--16]{Alperin}).
Also, for $p\not=7$, $\field[V]^{SL_2(\field_p)}$ is not Cohen-Macaulay and in fact has depth $3$ \cite[\S 5]{ShankWehlauDepth}.
In the language of L.E.~Dickson \cite[Lecture III\,\S 9]{Dickson}, we give a fundamental system for the formal modular invariants of the binary cubic.
Dickson considered this problem but was only able to identify a few specific invariants.
We proceed by constructing the required invariants and then proving that the given set generates $\field[V]^{SL_2(\field_p)}$. Our proof
relies on the construction of an infinite SAGBI basis and uses the Hilbert series calculation of
Hughes and Kemper \cite{HK}. 
Recall that a SAGBI basis is a {\bf S}ubalgebra {\bf A}nalog of a {\bf G}r\"obner {\bf B}asis for {\bf I}deals.
SAGBI bases were introduced independently by Robbiano-Sweedler \cite{rs} and Kapur-Madlener \cite{km}; 
a useful reference is Chapter~11 of Sturmfels \cite{Sturmfels} (who uses the term {\it canonical subalgebra basis}).
The ring of invariants of a modular representation of a $p$-group always has a finite SAGBI basis for an appropriate
choice of term order, see \cite{sw:cmipg}. A  finite SAGBI basis for the ring of invariants of the Sylow $p$-subgroup
of $\sltwo$ was computed in \cite{Shank1998}. Extensive preliminary calculations for small primes, using  MAGMA \cite{magma}, involving SAGBI bases and 
the relative transfer map, lead to the given generating set (see \cite{Hobson}). 
We use the graded reverse lexicographic order with $a_0<a_1<a_2<a_3$.
For background material on term orders and Gr\"obner bases see  Adams-Loustaunau \cite{AL}.
For background material on the invariant theory of finite groups
see Benson \cite{Benson}, Derksen-Kemper \cite{DK} or Neusel-Smith \cite{NS}.

A classical example of an invariant of a binary form is the discriminant, which in this case can be written as
\begin{equation*} D:=3a_2^2a_1^2-4a_3a_1^3-4a_2^3a_0+6a_3a_2a_1a_0-a_3^2a_0^2. \end{equation*}
Following Lecture III of  L.\,E.\,Dickson's  Madison Colloquium~\cite{Dickson} we identify the $SL_2(\field_p)$-invariant
\begin{equation*} L:=3(a_2^pa_1-a_2a_1^p)-(a_3^pa_0-a_3a_0^p).\end{equation*}

Let $B$ denote the Borel subgroup of $SL_2(\field_p)$ consisting of upper triangular matrices and let $P$ denote the unique
Sylow $p$-subgroup of $B$. Observe that $P$ is cyclic of order $p$ and is also a Sylow $p$-subgroup of $SL_2(\field_p)$.
Define 
$$N:=\prod_{\tau\in P}(a_3)\tau.$$
By Corollary~\ref{prodinv}, $N\cdot a_0$ is $SL_2(\field_p)$-invariant (or see \cite{Dickson}).

For a subgroup $H$ of a group $G$, choose coset representatives $G/H$ and define the
relative transfer
\begin{eqnarray*} \mbox{tr}^G_H: \field[V]^H &\rightarrow& \field[V]^G \\
f &\mapsto& \sum_{\tau \in G/H} (f)\tau. \end{eqnarray*}
The transfer, $\tr^G$, is the special case when $H$ is the trivial group. Define
$$K:=-\mbox{tr}^{SL_2(\field_p)}(a_1^{p-1}).$$
We show in Lemma~\ref{Klem} that $K$ is non-zero with lead monomial $a_2^{p-1}$.

For $\omega\in \field_p^*$, the diagonal matrix
$$\begin{array}{cccccc}\rho_{\omega}=\left[ \begin{array}{cc} \omega& 0\\ 0&\omega^{-1} \end{array}\right] & {\rm acts\ on}\ V^*\ {\rm as}&
\left[ \begin{array}{cccc} \omega^3&0&0&0\\ 0&\omega&0&0\\ 0&0&\omega^{-1}&0\\ 0&0&0&\omega^{-3} \end{array}\right].&&&
\end{array}
$$
This motivates the definition of a multiplicative weight function on monomials by
\begin{equation*} \mbox{wt}(a_i)=2i-3. \end{equation*}
Thus for any monomial $\beta$, we have $(\beta)\rho_{\omega}=\omega^{\mbox{wt}(\beta)}\beta$.
Since $\omega^{p-1}=1$, it is convenient to assume that the weight function takes values in
$\mathbb{Z}/(p-1)\mathbb{Z}$. Since $B$ is generated by elements of $P$ and $\rho_{\omega}$ for $\omega\in\field_p^*$,
it is clear that the $B$-invariants are precisely the isobaric $P$-invariants of weight zero (modulo $p-1$).

We show in Lemma~\ref{Isobaricprodtransferlem} that $N$ is isobaric of weight $3$ (modulo $p-1$). 
Let $c$ denote the smallest positive integer satisfying
$3c\equiv_{(p-1)}0$. Thus $c=(p-1)/3$ if $p\equiv_{(3)} 1$ and $c=p-1$ if $p\equiv_{(3)} -1$. 
Then $N^c$ is $B$-invariant and 
\begin{equation*} \delta:=\mbox{tr}_B^{SL_2(\mathbb{F}_p)}(N^c) \end{equation*} 
is $SL_2(\field_p)$-invariant. It follows from Theorem~\ref{orbitdiff} that the lead monomial of $\delta$
is $a_3^{pc}$. We show in Theorem~\ref{hsop} that 
$\{D,K,Na_0,\delta\}$ forms a homogeneous system of parameters, i.e.,
the set is algebraically independent and $\field[V]^{SL_2(\field_p)}$ is a finite module over $\field[D,K,Na_0,\delta]$.

It is easily verified that $d:=a_1^2-a_2a_0$ and $e:=2a_1^3+a_0(a_3a_0-3a_2a_1)$ are isobaric $P$-invariants of weight
$-2$ and $-3$ respectively. Define
$$\tilde{e}:=\mbox{tr}_B^{SL_2(\field_p)}(Ne).$$
We will show, see Theorem~\ref{maintheorem}, that for $p\equiv_{(3)}1$, the $SL_2(\field_p)$-invariants are generated by
$$D,K,L, Na_0,\delta,\tilde{e}$$
and an explicitly described finite subset of the image of the transfer.
For $p\equiv_{(3)}-1$ the additional invariant 
$$\tilde{d}:=\mbox{tr}_B^{SL_2(\field_p)}(N^{\frac{p+1}{3}}d)$$
is required.

\section{Preliminaries, lead monomials and t\^{e}te-\`{a}-t\^{e}tes}\label{prelim_sec}

For the remainder of the paper we use $G$ to denote $SL_2(\field_p)$. 
The following generalises \cite[2.4]{ShankWehlauDepth}.

\begin{Lem} If $f$ is an isobaric polynomial of weight $\lambda$, then $\tr^{P}(f)$ is isobaric of weight $\lambda$.
Furthermore $N$ is isobaric of weight $3$.
\label{Isobaricprodtransferlem}\end{Lem}

\begin{proof} The result follows from the fact that $P$ is normal in $B$. 
For $\omega \in \mathbb{F}_p^*$
\begin{eqnarray*} 
\left(\mbox{tr}^{P}\left(f\right)\right)\rho_{\omega}
&=&\sum_{\tau \in P} \left(f\right)\tau\rho_{\omega} 
=\sum_{\tau' \in P} \left(f\right)\rho_{\omega}\tau'\\
&=&\sum_{\tau' \in P} \omega^{\lambda}\left(f\right)\tau'
=\omega^{\lambda}\mbox{tr}^P(f).\end{eqnarray*} 
Thus $\tr^P(f)$ is isobaric of weight $\lambda$. A similar calculation gives $\wt(N)=\wt(a_3)=3$.
\end{proof}

Let $Q$ denote the subgroup generated by the transpose of $\sigma$, i.e., 
the lower triangular Sylow $p$-subgroup, and define
$$\eta:=\left[\begin{array}{cc} 0 & 1 \\ -1 & 0 \end{array}\right].$$

\begin{Lem} $Q\cup\{\eta\}$ is a set of coset representatives for $B$ in $SL_2(\mathbb{F}_p)$. \label{Bcosetreps} 
\end{Lem}

\begin{proof} Since the index of $B$ in $SL_2(\mathbb{F}_p)$ is $p+1$, we have the right number of elements.
To show that the cosets $(\sigma^T)^nB$ are distinct for $n=1,\ldots,p$, it is sufficient to show that
$(\sigma^T)^nB\not=B$ for $n<p$; this is clear. To show that $\eta B\not=(\sigma^T)^nB$, it is sufficient to show that
$\eta^{-1}(\sigma^T)^n\not\in B$; this is a straight forward calculation.
\end{proof}

\begin{Lem} $N a_0=-a_3\prod_{\tau\in Q}(a_0)\tau$.
\label{factor}
\end{Lem}
\begin{proof} Consider the orbits
$$a_3P=\{a_3+3sa_2+3s^2a_1+s^3a_0\mid s\in \field_p\}$$
and 
$$a_0Q=\{s^3a_3+3s^2a_2+3sa_1+a_0\mid s\in \field_p\}.$$
Thus 
\begin{eqnarray*}
N a_0&=&a_0\prod_{s\in \field_p}(a_3+3sa_2+3s^2a_1+s^3a_0)=a_0a_3\prod_{s\in\field_p^*}(a_3+3sa_2+3s^2a_1+s^3a_0)\\
&=&a_0a_3\prod_{s\in\field_p^*}s^3\left(\left(s^{-1}\right)^3a_3+3\left(s^{-1}\right)^2a_2+3s^{-1}a_1+a_0\right)\\
&=&a_3\left(\prod_{s\in\field_p^*}s^3\right)\prod_{\tau\in Q}\left(a_o\right)\tau=-a_3\prod_{\tau\in Q}\left(a_o\right)\tau
\end{eqnarray*}
\end{proof}

Since $\{\sigma,\sigma^T\}$ generates $SL_2(\field_p)$, any polynomial which is both $P$-invariant and $Q$-invariant is
$SL_2(\field_p)$-invariant, giving the following corollary (see also Lecture III \S 9 of \cite{Dickson}).

 \begin{Cor}  $N a_0$ is $SL_2(\field_p)$-invariant.
\label{prodinv} 
\end{Cor}

\begin{The} Suppose $f$ is an isobaric $P$-invariant with $\wt(N\cdot f)=0$. Then $a_0$ divides $\tr_B^G(N\cdot f)-N\cdot f$
and, if $a_0$ does not divide $f$, the lead terms of $\tr_B^G(N\cdot f)$ and $N\cdot f$ are equal.
\label{orbitdiff}\end{The}

\begin{proof} Using the fact that $N a_0$ is $SL_2(\field_p)$-invariant we see that
\begin{eqnarray*}\mbox{tr}_B^G(N\cdot f)-N\cdot f&=&N a_0\left(\mbox{tr}_B^G\left(fa_0^{-1}\right)\right)-N\cdot f \\
&=&N\left(a_0\mbox{tr}_B^G\left(fa_0^{-1}\right)-f\right).\end{eqnarray*}
Observe that $(a_0)\eta=-a_3$. Thus, using the coset representatives from Lemma~\ref{Bcosetreps}, we have
\begin{equation*} 
a_0\mbox{tr}_B^G\left(fa_0^{-1}\right)-f= a_0\left(\sum_{\tau \in Q\setminus\{1\}} \frac{(f)\tau}{(a_0)\tau} -\frac{\left(f\right)\eta}{a_3}\right). \end{equation*}
From Lemma~\ref{factor}, $N$ is a least common multiple of
$\{a_3\}\cup \{(a_0)\tau\mid\tau \in Q\setminus \{1\}\}$. Taking $N$ as the common denominator in the above sum gives
$$a_0\mbox{tr}_B^G(fa_0^{-1})-f=\frac{a_0J}{N}$$
for some polynomial $J$. Therefore  $\mbox{tr}_B^G(N\cdot f)-N\cdot f=a_0J$. If $a_0$ does not divide $f$, then the lead term of $N\cdot f$ is not divisible by $a_0$
and is also the lead term of  $\mbox{tr}_B^G(N\cdot f)$.
\end{proof}

We use LM to denote lead monomial and LT to denote lead term. It is clear that $\lm(N)=a_3^p$.
In the following lemmas, we use the lead monomial calculations from \cite{Shank1998}. 
Note that although the basis used in \cite{Shank1998} is different from the one used here, the change of basis is upper triangular 
and so the lead monomial calculations still apply.

\begin{Lem}  For $m=2+\lfloor 3j/(p-1)\rfloor$,
 $$\lm\left({\rm tr}_B^G\left(N^j{\rm tr}^{P}\left(a_2^{(m-1)(p-1)-3j}a_3^{p-1}\right)\right)\right)=a_3^{pj}a_2^{m(p-1)-3j}=:\gamma_j.$$
\label{gammafamily} \end{Lem} \vspace{-0.9cm}
\begin{proof} We know from \cite[3.3]{Shank1998} that tr$^{P}(a_2^ba_3^{p-1})$ has lead monomial $a_2^{b+p-1}$ if $1 \leq b \leq p-1$. 
Since $m=2+\lfloor 3j/(p-1)\rfloor$, we have $3j/(p-1)-1<m-2\leq 3j/(p-1)$, which simplifies to
$0<(m-1)(p-1)-3j\leq p-1$. The result then follows from 
Lemma~\ref{Isobaricprodtransferlem} and Theorem~\ref{orbitdiff}.
\end{proof}

\begin{Lem} For $0\leq j\leq (p-1)/2$, 
\begin{equation*} \lm\left({\rm tr}_B^G\left(N^j{\rm tr}^{P}\left(a_3^{p-1-j}\right)\right)\right)=a_3^{pj}a_2^{p-1-2j}a_1^j=:\beta_j. \end{equation*}
\label{pm1amod2}\end{Lem} \vspace{-1cm}
\begin{proof} From \cite[3.2]{Shank1998}, tr$^{P}(a_3^b)$ has lead monomial $a_2^{2b-(p-1)}a_1^{p-1-b}$ if $(p-1)/2 \leq b \leq p-1$. 
Simplifying $(p-1)/2\leq p-1-j\leq p-1$ gives $0\leq j \leq (p-1)/2$. The result then follows from 
Lemma~\ref{Isobaricprodtransferlem} and Theorem~\ref{orbitdiff}.
\end{proof}

\begin{Lem}  For $m=2+\lfloor 3j/(p-1)\rfloor$ and $j\not=\lceil(m-2)(p-1)/3\rceil$,
\begin{equation*} \lm\left({\rm tr}_B^G\left(N^j{\rm tr}^{P}\left(a_3^{p-2}a_2^{(m-1)(p-1)+3-3j}\right)\right)\right)=a_3^{pj}a_2^{m(p-1)+1-3j}a_1=:\Delta_j. \end{equation*}
\label{Deltafam} \end{Lem} \vspace{-0.9cm}
\begin{proof} Using \cite[3.4]{Shank1998}, $\lm(\tr^{P}(a_3^{p-2}a_2^b))=a_2^{b+p-3}a_1$ for $2 \leq b \leq p-1$. As in the proof of 
Lemma~\ref{gammafamily}, we have
$0<(m-1)(p-1)-3j\leq p-1$. Therefore $3<(m-1)(p-1)+3-3j\leq p+2$. Thus the lead monomial calculation is valid as long as 
$(m-1)(p-1)+3-3j\not\in\{p,p+1,p+2\}$. This simplifies to $j\not\in\{(m-2)(p-1)/3 +\varepsilon/3\mid\varepsilon\in\{0,1,2\}\}$, i.e.,
$j\not=\lceil(m-2)(p-1)/3\rceil$.
 The result then follows from 
Lemma~\ref{Isobaricprodtransferlem} and Theorem~\ref{orbitdiff}.
\end{proof}

\begin{Lem} For $p\equiv_{(3)}-1$ and $j=(2p-1)/3,\ldots,p-2$,
\begin{equation*} \lm\left({\rm tr}_B^G\left(N^j{\rm tr}^{P}\left(a_3^{\frac{5p-7}{3}-j}a_2^2\right)\right)\right)=a_3^{pj}a_2^{\frac{7p-5}{3}-2j}a_1^{j-\frac{2p-4}{3}}=:\phi_j.\end{equation*}
\label{pm1amod4} \end{Lem} \vspace{-0.9cm}
\begin{proof}
From \cite[3.5]{Shank1998}, $\lm(\tr^{P}(a_3^ba_2^2))=a_2^{2b-p+3}a_1^{p-1-b}$ for $(p-2)/{3}\leq b \leq p-1$.
The inequalities $(p-2)/{3}\leq (5p-7)/3 -j \leq p-1$ simplify to
$(2p-4)/3\leq j \leq (7p-5)/6=p-1+(p+1)/6$. Thus the lead monomial calculation is valid for the given range of $j$.
The result then follows from 
Lemma~\ref{Isobaricprodtransferlem} and Theorem~\ref{orbitdiff}.
\end{proof}

Define $\xi=3a_2^2-4a_3a_1$.

\begin{Lem}  $K=-\tr^P(a_3^{p-1})-a_0^{p-1} \equiv_{(a_0)}(3\xi)^{\frac{p-1}{2}}+a_1^{p-1}$.\label{Klem}
\end{Lem}
\begin{proof} A simple calculation gives $\tr^P(a_1^{p-1})=-a_0^{p-1}$ (or see \cite[3.2]{Shank1998}). Since $\wt(a_0^{p-1})=0$ and the index of $P$ in $B$ is $p-1$,
we have $\tr^B(a_1^{p-1})=a_0^{p-1}$. Using the coset representatives from Lemma~\ref{Bcosetreps} gives
\begin{eqnarray*}
-K&=&{\rm tr}^G(a_1^{p-1})={\rm tr}_B^G(a_0^{p-1})=((a_0)\eta)^{p-1}+{\rm tr}^Q(a_0^{p-1})=a_3^{p-1}+{\rm tr}^Q(a_0^{p-1})\\
&=&a_3^{p-1}+\sum_{s\in\field_p}(s^3a_3+3s^2a_2+3sa_1+a_0)^{p-1}\\
&=&a_3^{p-1}+a_0^{p-1}+\sum_{s\in\field_p^*}(s^3a_3+3s^2a_2+3sa_1+a_0)^{p-1}\\
&=&a_3^{p-1}+a_0^{p-1}+\sum_{s\in\field_p^*}s^{3(p-1)}(a_3+3s^{-1}a_2+3(s^{-1})^2a_1+(s^{-1})^3a_0)^{p-1}\\
&=&a_0^{p-1}+\sum_{t\in\field_p}(a_3+3ta_2+3t^2a_1+t^3a_0)^{p-1}=a_0^{p-1}+{\rm tr}^P(a_3^{p-1})\\
&\equiv_{(a_0)}&\sum_{t\in\field_p}(a_3+3ta_2+3t^2a_1)^{p-1}\\
&\equiv_{(a_0)}&\sum_{t\in\field_p}\sum_{a+b+c=p-1}\binom{p-1}{a,b,c}t^{b+2c}a_3^a(3a_2)^b(3a_1)^c.
\end{eqnarray*}
It is well known that $\sum_{t\in\field_p}t^i$ is $-1$ if $i$ is a positive multiple of $p-1$ and $0$ otherwise.
Thus, for $a,b,c$ non-negative with $a+b+c=p-1$, we see that $\sum_{t\in\field_p}t^{b+2c}$ is non-zero only when $b+2c=p-1$ or $b+2c=2(p-1)$.
If $b+2c=2(p-1)$ then $c=p-1$ and $a=b=0$. If $b+2c=p-1$ then $a=c$. Therefore
\begin{eqnarray*}
-K&\equiv_{(a_0)}& \binom{p-1}{0,0,p-1}(-1)(3a_1)^{p-1}-\sum_{c=0}^{\frac{p-1}{2}}\binom{p-1}{c,b,c}(3a_2)^{p-1-2c}(3a_1a_3)^c\\
-K&\equiv_{(a_0)}& -a_1^{p-1}-3^{\frac{p-1}{2}}\sum_{c=0}^{\frac{p-1}{2}}\binom{p-1}{c,b,c}(3a_2^2)^{\frac{p-1}{2}-c}(a_1a_3)^c.
\end{eqnarray*}
Simplifying binomial coefficients modulo $p$ gives
$$\binom{p-1}{c,p-1-2c,c}=\binom{2c}{c}=(-4)^c\binom{\frac{p-1}{2}}{c}.$$
Thus 
$$K\equiv_{(a_o)}a_1^{p-1}+3^{\frac{p-1}{2}}(3a_2^2-4a_1a_3)^{\frac{p-1}{2}},$$
as required.
\end{proof}

A similar calculation using the identity
$$\binom{p-2}{a,p-3-2a,a+1}\equiv_{(p)}-2(a+1)\binom{2a+1}{a}\equiv_{(p)}-2(-4)^a\binom{\frac{p-3}{2}}{a}$$
gives the following lemma.

\begin{Lem}$\tr^P(a_3^{p-2})\equiv_{(a_0)} 6a_1(3\xi)^{\frac{p-3}{2}}$.
\label{Ptrlem}
\end{Lem}

\begin{The}\label{hsop} The set $\{D,K,Na_0,\delta\}$ is a homogeneous system of parameters.
\end{The}
\begin{proof}
With out loss of generality, we may assume $\field$ is algebraically closed.
We will show that the variety associated to $(D,K,Na_0,\delta)\field[V]$,
say $\mcV$, consists of the zero vector.

Suppose $v\in\V$. Since $Na_0(v)=0$, there exits $g\in SL_2(\field_p)$ such that
$a_0g(v)=0$. Replacing $v$ with $g(v)$ if necessary, we may assume $a_0(v)=0$.
Note that $D\equiv_{(a_0)}a_1^2\xi$. From Lemma~\ref{Klem}, $K\equiv_{(a_0)}(3\xi)^{\frac{p-1}{2}}+a_1^{p-1}$.
Thus $a_1^2K-3(3\xi)^{\frac{p-3}{2}}D\equiv_{(a_0)}a_1^{p+1}$. Therefore $a_1(v)=0$.
Since $\lm(K)=a_2^{p-1}$ in the grevlex order, we have $a_2(v)=0$.
Since $\lm(\delta)=a_3^{pc}$, we have $a_3(v)=0$. Therefore $v$ is the zero vector.
\end{proof}

If $f$ and $h$ are polynomials with $\lt(f)=\lt(h)$, we refer to $f-h$ as a  t\^{e}te-\`{a}-t\^{e}tes 
(see \cite{rs} or \cite{Shank1998}).

\begin{The} There is an infinite family of t\^{e}te-\`{a}-t\^{e}tes in $\mathbb{F}[V]^{SL_2(\mathbb{F}_p)}$, defined as follows:
\begin{eqnarray*} h_1&=&K\cdot {\rm tr}_B^{SL_2(\mathbb{F}_p)}(Ne)-D\cdot{\rm tr}_B^{SL_2(\mathbb{F}_p)}(N{\rm tr}^{P}(a_3^{p-2})), \\
h_2&=&K\cdot h_1-(3D)^{\frac{p-1}{2}}\cdot {\rm tr}_B^{SL_2(\mathbb{F}_p)}(Ne), \\
h_i&=&K\cdot h_{i-1} - (3D)^{\frac{p-1}{2}}\cdot h_{i-2} \;\, {\rm for} \, i \geq 3, \end{eqnarray*}
with $\lt(h_i)=2a_3^pa_1^{p+2+(i-1)(p-1)}$ for $i\geq1$.
\label{inffamthm} \end{The}
\begin{proof} The proof is by induction on $i$. Recall that $\lt(D)=3a_1^2a_2^2$. From Lemma~\ref{Klem},
$\lt(K)=a_2^{p-1}$. Using Theorem~\ref{orbitdiff} and Lemma~\ref{Ptrlem}, we have $\lt (\tr_B^G(N\tr^P(a_3^{p-2}))=\frac{2}{3}a_1a_2^{p-3}a_3^p$
and $\lt(\tr_B^G(Ne))=2a_1^3a_3^p$. Thus $h_1$ is indeed a t\^{e}te-\`{a}-t\^{e}te. Since $\lt((3D)^{(p-1)/2})=(a_1a_2)^{p-1}$, it is sufficient to prove
$\lt(h_i)=2a_3^pa_1^{p+2+(i-1)(p-1)}$ for $i\geq1$.

Define
\begin{eqnarray*} r_1&=&K\cdot e - D\cdot \mbox{tr}^{P}(a_3^{p-2}),\\
r_2&=&K\cdot r_1 - (3D)^{\frac{p-1}{2}} \cdot e,\\
r_i&=&K\cdot r_{i-1} - (3D)^{\frac{p-1}{2}} \cdot r_{i-2}\ {\rm for}\ i \geq 3. \end{eqnarray*}
Since $K$ and $D$ are $G$-invariant, we have $h_i=\tr_B^G(Nr_i)$. Thus, using Theorem~\ref{orbitdiff}, it is sufficient to prove
$\lt(r_i)=2a_1^{p+2+(i-1)(p-1)}$ for $i\geq 1$.

Note that $e\equiv_{(a_0)}2a_1^3$ and $D\equiv_{(a_0)}a_1^2\xi$. Thus, using Lemma~\ref{Klem} and Lemma~\ref{Ptrlem},
$$ r_1\equiv_{(a_0)}((3\xi)^{\frac{p-1}{2}}+a_1^{p-1})\cdot 2a_1^3-a_1^2\xi\cdot 2(3^{\frac{p-1}{2}})a_1\xi^{\frac{p-3}{2}}=2a_1^{p+2}.$$
Similarly
$$r_2\equiv_{(a_0)}((3\xi)^{\frac{p-1}{2}}+a_1^{p-1})\cdot 2a_1^{p+2}-(3a_1^2\xi)^{\frac{p-1}{2}}\cdot 2a_1^3=2a_1^{(p+2)+(p-1)}.$$
Using the induction hypothesis,
\begin{eqnarray*}
r_i&\equiv_{(a_0)}&((3\xi)^{\frac{p-1}{2}}+a_1^{p-1})\cdot 2a_1^{p+2+(i-2)(p-1)}-(3a_1^2\xi)^{\frac{p-1}{2}}\cdot 2a_1^{p+2+(i-3)(p-1)}\\
&\equiv_{(a_0)}& 2a_1^{p+2+(i-1)(p-1)},
\end{eqnarray*}
as required.
\end{proof}

\section{Generators and Hilbert series}\label{MAINtheorem}

This section is devoted to the proof of the main theorem. 

\begin{The}
For $p>3$, $\mathbb{F}[V]^{SL_2(\mathbb{F}_p)}$ is generated by
 
\begin{itemize}
\item elements from the image of the transfer 
\item $D, K, L, \delta, Na_0, \tilde{e}$ and 
\item for $p \equiv -1$ mod 3, $\tilde{d}$.
\end{itemize}

\noindent The generators from the image of the transfer fall into three families:

\begin{enumerate} 
\item $\mbox{tr}^{SL_2(\mathbb{F}_p)}(N^ja_2^{(m-1)(p-1)-3j}a_3^{p-1})$ where

\begin{equation*} j=\left\{\begin{array}{ll} 1,\ldots,(p-4)/3 & \mbox{for} \,\, p\equiv 1 \,\, \mbox{mod} \,\, 3\\
1,\ldots,p-2 & \mbox{for} \,\, p\equiv -1 \,\, \mbox{mod} \,\, 3\end{array}\right.\end{equation*}
and $m=2+\lfloor 3j/(p-1)\rfloor$;
\item $\mbox{tr}^{SL_2(\mathbb{F}_p)}(N^ja_3^{p-1-j})$ where

\begin{equation*} j=\left\{\begin{array}{ll}1,\ldots,(p-4)/3 & \mbox{for} \,\, p \equiv 1 \,\, \mbox{mod} \,\, 3\\
1,\ldots,(p-2)/3 & \mbox{for} \,\, p \equiv -1 \,\, \mbox{mod} \,\, 3;\end{array}\right. \end{equation*}

\item and $\mbox{tr}^{SL_2(\mathbb{F}_p)}(N^ja_3^{p-2}a_2^{(m-1)(p-1)+3-3j})$ where

\begin{equation*} j=\left\{\begin{array}{ll} 2,\ldots,(p-4)/3 & \mbox{for} \,\, p \equiv 1 \,\, \mbox{mod} \,\, 3\\
2,\ldots,p-2 \ {\rm with}\ j \neq (p+1)/3, (2p-1)/3 & \mbox{for} \,\, p \equiv -1 \,\, \mbox{mod} \,\, 3 \end{array}\right.\end{equation*}
and $m=2+\lfloor 3j/(p-1)\rfloor$.
\end{enumerate}

\noindent For $p \equiv -1$ mod 3, we have the further family of invariants:

\begin{equation*} \mbox{tr}^{SL_2(\mathbb{F}_p)}(N^ja_3^{\frac{5p-7-3j}{3}}a_2^2), \;\; j=\frac{2p-1}{3},\ldots,p-2. \end{equation*}

\label{maintheorem} \end{The}

\hbox{}

Let $\mcC$ denote the proposed generating set and let $R$ denote
the algebra generated by $\mcC$. Since the elements of $\mcC$ are homogeneous invariants,
$R$ is a graded subalgebra of $\slinv$. Recall that the {\it Hilbert Series} of a graded vector
space $M=\oplus_{\ell=0}^{\infty}M_{\ell}$ is the formal power series 
$HS(M,t)=\sum_{\ell=0}^{\infty}{\rm dim}(M_{\ell})t^{\ell}$. Since $R$ is a graded subalgebra
of $\slinv$, we have $HS(R,t)\leq HS(\slinv,t)$. We  prove the theorem by showing these series
are equal.

Define $\mcG:=\mcC\cup\{h_i, \, \forall i\geq 1\}$ and let
$\lt(\mcG)$ denote the subalgebra generated by the lead monomials of the elements of $\mcG$.  
In each of the two cases, $p \equiv 1$ mod 3 and $p \equiv -1$ mod 3, we
choose a graded subspace $Z$ of $\lt(\mcG)$, giving a chain of inequalities:
\begin{equation*} HS(Z,t) \leq HS(LT(\mcG),t) \leq HS(LT(R),t)= HS(R,t) 
\leq HS(\slinv,t). \end{equation*}
We calculate $HS(Z,t)$ and  compare with  Hughes-Kemper \cite{HK} to show
$HS(Z,t)=HS(\slinv,t)$. This proves that $\mcC$ is a generating set and $\mcG$ is a SAGBI basis.

The invariants $D,K,N a_0$, and $\delta$ have lead monomials 
$LM(D)=a_2^2a_1^2$, $LM(K)=a_2^{p-1}$, $LM(N a_0)=a_3^pa_0$ and $LM(\delta)=a_3^{pc}$, 
where $c=(p-1)/3$ if $p \equiv_{(3)} 1$ and $a=p-1$ if $p \equiv_{(3)} -1$. 
Define
\begin{equation*} A:=\mathbb{F}[a_2^2a_1^2, a_2^{p-1}, a_3^pa_0, a_3^{pc}], \end{equation*}
the algebra generated by $LM(D),LM(K),LM(N a_0)$ and $LM(\delta)$. In each of the two cases we will define $Z$
as an $A$  - submodule of $\lt(\mcG)$. For a monomial $a_3^{e_3}a_2^{e_2}a_1^{e_1}a_0^{e_0}$ we assign a {\it parity}
$(e_2\, {\rm mod}\, 2, e_1\, {\rm mod}\, 2)$ and observe that the action of $A$ preserves parity.

\noindent \textbf{The $p \equiv 1$ mod 3 Case}

Recall from Theorem~\ref{inffamthm} that the lead monomials of  the t\^{e}te-\`{a}-t\^{e}tes 
$h_i$ are $LM(h_i)=a_3^pa_1^{p+2+(i-1)(p-1)}$ for $i\geq1$. 
By Lemma~\ref{orbitdiff} the lead monomial of the invariant $\tilde{e}=\tr_B^{SL_2(\mathbb{F}_p)}(N e)$ 
is equal to $a_3^pa_1^3$. Hence we have

\begin{equation*} n_i:=a_3^pa_1^{3+i(p-1)} \; \mbox{for} \; i\geq 0 \end{equation*}
as the lead monomials of  $\tilde{e}$ and $h_i$. Denote

$$ \alpha_{ij}:=n_0^{j-1}n_i
=a_3^{pj}a_1^{3j+(p-1)i}, \;\; 1\leq j\leq (p-1)/3, \;\; i\geq 0 $$
and
$$
\epsilon_{ij}:=LM(L)\alpha_{ij}=a_3^{pj}a_2^pa_1^{1+3j+(p-1)i}, \;\; 1\leq j\leq (p-1)/3, \;\; i\geq 0.
$$
Define $Z$ to be the $A$ - module generated by the monomials
$$
\mcB:=\left\{1, LM(L), \gamma_j, \beta_j, \Delta_j, \alpha_{ij}, \epsilon_{ij}
\mid i\in\bbN\right\}. 
$$
where $1\leq j\leq (p-1)/3$ for the $\alpha$ and $\epsilon$ families, $1\leq j<(p-1)/3$ for the $\gamma$ and $\beta$ families,
and $1<j<(p-1)/3$ for the $\Delta$ family; see Lemma~\ref{gammafamily}, Lemma~\ref{pm1amod2} and Lemma~\ref{Deltafam}
for the definition of $\gamma_j$, $\beta_j$ and $\Delta_j$, and compare
with the range of $j$ for the families of transfers in Theorem~\ref{maintheorem}.

The action of $LM(Na_0)$ and
 $LM(\delta)$ on $Z$ is essentially free: every monomial in $Z$ with a factor of
 $a_0^{e_0}$ is divisible by $LM(Na_0)^{e_0}$ and the remaining power of $a_3$
determines the power of $\lm(\delta)$. 
Let $\tZ$ denote the span of the monomials in $Z$ which are reduced with respect to  $\lm(Na_0)$ and $\lm(\delta)$.
Then 
$$HS(Z,t)=\frac{HS(\tZ,t)}{(1-t^{p+1})(1-t^{p(p-1)/3})}.$$

Define $\tZ_j$ to be the span of the monomials in $\tZ$ of the form $a_3^{pj}a_2^{e_2}a_1^{e_1}$.
Then $$\tZ=\bigoplus_{j=0}^{(p-1)/3}\tZ_j.$$ We proceed by computing $HS(\tZ_j,t)$ for 
$j=0,1,\ldots,(p-1)/3$. For fixed $j$, we determine the monomials
$a_3^{pj}a_2^{x}a_1^{y}\in \tZ_j$. This set can be identified with a subset
of the integral lattice in the $xy$-plane. Each element of $\mcB$ gives rise
to a $\field[\lm(D),\lm(K)]$-submodule corresponding to a cone in the $xy$-plane. 
The monomials in $\tZ_j$ correspond to the union of these cones. The cones
corresponding to elements of $\mcB$ of different parity are disjoint.

For $j=0$,
the only elements of $\mcB$ are 1 and $LM(L)=a_2^pa_1$, of parity $(0,0)$ and $(1,1)$ respectively.
Thus 
$$HS(\tZ_0,t)=\frac{1+t^{p+1}}{(1-t^4)(1-t^{p-1})}.$$

For $j=(p-1)/3=c$, the elements of $\mcB$ fall into two families:
\begin{itemize}
\item $\alpha_{ic}= a_3^{pc}a_1^{p-1+i(p-1)}$ for $i\in\bbN$, with  parity $(0,0)$;
\item $\epsilon_{ic}=a_3^{pc}a_2^pa_1^{p+i(p-1)}$ for $ i\in\bbN$, with parity $(1,1)$.
\end{itemize}

For parity $(0,0)$: Note that $\alpha_{0c}\lm(K)=\lm(\delta)\lm(D)^{\frac{p-1}{2}}\not\in\tZ$.
Furthermore, for $i>0$, we have $\alpha_{ic}\lm(K)=\alpha_{i-1,c}\lm(D)^{\frac{p-1}{2}}$. Thus it is sufficient
to count the monomials $\alpha_{ic}\lm(D)^{\ell}$ with $i,\ell\in\bbN$.

For parity $(1,1)$: Note that $\epsilon_{0c}\lm(K)=\lm(\delta)\lm(L)\lm(D)^{\frac{p-1}{2}}\not\in\tZ$.
Furthermore, for $i>0$, we have $\epsilon_{ic}\lm(K)=\epsilon_{i-1,c}\lm(D)^{\frac{p-1}{2}}$. Thus it is sufficient
to count the monomials $\epsilon_{ic}\lm(D)^{\ell}$ with $i,\ell\in\bbN$.

Counting monomials and identifying the appropriate geometric series gives
$$HS(\tZ_c,t)=\frac{t^{pc}(t^{p-1}+t^{2p})}{(1-t^4)(1-t^{p-1})}=\frac{t^{pc+p-1}(1+t^{p+1})}{(1-t^4)(1-t^{p-1})}.$$

In the case $j=1$, we have the following elements of $\mcB$:
\begin{itemize}
\item $\alpha_{i1}= a_3^pa_1^{3+i(p-1)}$ for $i\in\bbN$, with parity $(0,1)$;
\item $\beta_1=a_3^pa_2^{p-3}a_1$, with parity $(0,1)$;
\item $\gamma_1=a_3^pa_2^{2p-5}$, with parity $(1,0)$;
\item $\epsilon_{i1}=a_3^pa_2^pa_1^{4+i(p-1)}$ for $i\in\bbN$, with parity $(1,0)$.
\end{itemize}

For Parity $(0,1)$: Since $\alpha_{01}\lm(K)=\beta_1\lm(D)$ and $\alpha_{i1}\lm(K)=\alpha_{i-1,1}\lm(D)^{\frac{p-1}{2}}$, 
for $i>0$, it is sufficient to count the  monomials $\alpha_{i1}\lm(D)^{\ell}$ and $\beta_1\lm(K)^i\lm(D)^{\ell}$.

For Parity $(1,0)$: Since $\epsilon_{01}\lm(K)=\gamma_1\lm(D)$ and $\epsilon_{i1}\lm(K)=\epsilon_{i-1,1}\lm(D)^{\frac{p-1}{2}}$, 
for $i>0$, it is sufficient to count the monomials $\epsilon_{i1}\lm(D)^{\ell}$ and $\gamma_1\lm(K)^i\lm(D)^{\ell}$.

Counting monomials and identifying the appropriate geometric series gives
$$HS(\tZ_1,t)=\frac{t^{p}(t^3+t^{p-2}+t^{p+4}+t^{2p-5})}{(1-t^4)(1-t^{p-1})}.$$

We now consider the case where $j=2k$
is even and $2\leq j<\frac{p-1}{3}$.
The relevant  monomials are:
\begin{itemize}
\item $\alpha_{ij}= a_3^{pj}a_1^{3j+i(p-1)}$ for $i\in\bbN$, with parity $(0,0)$;
\item $\beta_j=a_3^{pj}a_2^{p-1-2j}a_1^j$, with parity $(0,0)$;
\item $\gamma_j=a_3^{pj}a_2^{2p-2-3j}$, with parity $(0,0)$;
\item $\Delta_j=a_3^{pj}a_2^{2p-1-3j}a_1$, with parity $(1,1)$;
\item $\epsilon_{ij}=a_3^{pj}a_2^pa_1^{3j+1+i(p-1)}$ for $i\in\bbN$, with parity $(1,1)$.
\end{itemize}
For parity $(0,0)$: Observe that $\beta_j\lm(K)=\gamma_j\lm(D)^k$, $\alpha_{0j}\lm(K)=\beta_j\lm(D)^j$
and $\alpha_{ij}\lm(K)=\alpha_{i-1,j}\lm(D)^{(p-1)/2}$ for $i>0$. Thus it is sufficient to count
the monomials $\alpha_{ij}\lm(D)^{\ell}$, $\beta_j\lm(D)^{\ell}$ and $\gamma_j\lm(D)^{\ell}\lm(K)^i$, for
$i,\ell\in\bbN$.

For parity $(1,1)$:  Since $\epsilon_{0j}\lm(K)=\Delta_j\lm(D)^{3k}$ and 
$\epsilon_{ij}\lm(K)=\epsilon_{i-1,j}\lm(D)^{\frac{p-1}{2}}$ 
for $i>0$, it is sufficient to count the monomials $\epsilon_{ij}\lm(D)^{\ell}$ and $\Delta_j\lm(K)^i\lm(D)^{\ell}$.

Counting monomials and identifying the appropriate geometric series gives
$$HS(\tZ_{2k},t)=t^{2kp}\left(\frac{t^{6k}+t^{2p-2-6k}+t^{p+6k+1}+t^{2p-6k}}{(1-t^4)(1-t^{p-1})}
+\frac{t^{p-1-2k}}{1-t^4}\right)$$
for  $k=1,\ldots,\frac{p-7}{6}$. 

For $j=2k+1$ odd with $1<j<(p-1)/3$, the elements of $\mcB$ are:
\begin{itemize}
\item $\alpha_{ij}= a_3^{pj}a_1^{3j+i(p-1)}$ for $i\in\bbN$, with parity $(0,1)$;
\item $\beta_j=a_3^{pj}a_2^{p-1-2j}a_1^j$ with parity $(0,1)$;
\item $\Delta_j=a_3^{pj}a_2^{2p-1-3j}a_1$ with parity $(0,1)$;
\item $\gamma_j=a_3^{pj}a_2^{2p-2-3j}$ with parity $(1,0)$;
\item $\epsilon_{ij}=a_3^{pj}a_2^pa_1^{3j+1+i(p-1)}$ for $i\in\bbN$, with parity $(1,0)$.
\end{itemize}

For parity $(0,1)$:  Observe that $\beta_j\lm(K)=\Delta_j\lm(D)^k$, $\alpha_{0j}\lm(K)=\beta_j\lm(D)^j$
and $\alpha_{ij}\lm(K)=\alpha_{i-1,j}\lm(D)^{(p-1)/2}$ for $i>0$. Thus it is sufficient to count
the monomials $\alpha_{ij}\lm(D)^{\ell}$, $\beta_j\lm(D)^{\ell}$ and $\Delta_j\lm(D)^{\ell}\lm(K)^i$, for
$i,\ell\in\bbN$.

For parity $(1,0)$:   Since $\epsilon_{0j}\lm(K)=\gamma_j\lm(D)^{3k}$ and 
$\epsilon_{ij}\lm(K)=\epsilon_{i-1,j}\lm(D)^{\frac{p-1}{2}}$ 
for $i>0$, it is sufficient to count the monomials $\epsilon_{ij}\lm(D)^{\ell}$ and $\gamma_j\lm(K)^i\lm(D)^{\ell}$.

Counting monomials and identifying the appropriate geometric series gives
$$
HS(\tZ_{2k+1},t)=t^{(2k+1)p}\left(\frac{t^{6k+3}+t^{2p-2-6k-3}+t^{p+6k+4}+t^{2p-6k-3}}{(1-t^4)(1-t^{p-1})}
+\frac{t^{p-2-2k}}{1-t^4}\right)
$$
for  $k=1,\ldots,\frac{p-7}{6}$. 

The even and odd formulae can be put in a common form: for $1<j<(p-1)/3$,
$$HS(\tZ_j,t)=\frac{t^{jp}\left(t^{3j}+t^{2p-2-3j}+t^{p+1+3j}+t^{2p-3j}+t^{p-1-j}(1-t^{p-1})\right)}{(1-t^4)(1-t^{p-1})}.$$

Summing over $j$ and simplifying gives

\begin{equation*} HS(Z,t)=\frac{Numer(t)}{Denom(t)} \end{equation*}

\noindent where

\begin{eqnarray*} Numer(t)&=&(1+t^{p+1}+t^{p+3}+t^{2p-2}+t^{2p+4}+t^{3p-5}+t^{p-1}(t^{2p-2}-t^{(p-1)(p-1)/3})\\
&+&t^{\frac{p(p-1)}{3}+p-1}+t^{\frac{p(p-1)}{3}+2p})(1-t^{p-3})(1-t^{p+3})\\
&+&(t^{2p-2}+t^{2p})(t^{2p-6}- t^{(p-3)(p-1)/3})(1-t^{p+3})\\
&+&(1+t^{p+1})(t^{2p+6}-t^{(p+3)(p-1)/3})(1-t^{p-3})\end{eqnarray*}

\noindent and

\begin{equation*} Denom(t)=(1-t^4)(1-t^{p-3})(1-t^{p-1})(1-t^{p+1})(1-t^{p+3})(1-t^{\frac{p(p-1)}{3}}).\end{equation*}

\noindent This agrees with the calculation of $HS(\field[V]^G,t)$ by Hughes-Kemper \cite[2.7(d)]{HK}.

\hbox{}

\noindent \textbf{The $p \equiv -1$ mod $3$ Case}

In this case the lead monomial of $\delta=\tr_B^G(N^c)$ is $a_3^{p(p-1)}$ and the generators of $Z$ will be monomials
divisible by $a_3^{pj}$ for $j\leq p-1$. Using Lemma~\ref{orbitdiff} the lead monomial of $\tilde{d}$ is
$a_3^{(p+1)/3}a_1^2$. As in the proof of the $p\equiv_{(3)} 1$ case, we denote the lead monomials of $\tilde{e}$ and
$h_i$ by $n_i=a_3^pa_1^{3+i(p-1)}$ for $i\geq 0$. Define $s:=\lfloor 3j/(p-1)\rfloor$,

$$\alpha_{ij}:=\lm(\tilde{d})^s n_i n_0^{j-1-s(p-1)/3}=a_3^{pj}a_1^{3j+(p-1)(i-s)},
 \;\; 1\leq j\leq (p-1), \;\; i\in\bbN $$
and
$$
\epsilon_{ij}:=\lm(L)\alpha_{ij}=a_3^{pj}a_2^pa_1^{3j+(p-1)(i-s)+1}, \;\; 1\leq j\leq (p-1), \;\; i\in\bbN.
$$

\noindent Further, we assign the following notation:

\begin{eqnarray*} 
\lambda &:=& \lm(\tilde{d})\gamma_{\frac{p-2}{3}}=a_3^{p\frac{2p-1}{3}}a_2^pa_1^2, \\
\mu     &:=&\beta_1 \cdot \gamma_{\frac{p-2}{3}}=a_3^{p\frac{p+1}{3}}a_2^{2p-3}a_1, \\
\eta_j&:=&\lm(\tilde{d})\beta_{j-(p+1)/3}=a_3^{pj}a_2^{\frac{5p-1}{3}-2j}a_1^{j-\frac{p-5}{3}}
 \;\;\; \mbox{for} \;\; \frac{p+4}{3}\leq j\leq\frac{2p-1}{3}.
\end{eqnarray*}

Define $Z$ to be the $A$ -- module generated by 
\begin{equation*}\mcB:=\{1, LM(L), \alpha_{i,j}, \epsilon_{i,j}, \gamma_j, \beta_j, \Delta_j, \phi_j, \lambda, \mu, \eta_j
\mid i\in\bbN\}. \end{equation*}
where the ranges in $j$ are given above or in the statement of Theorem~\ref{maintheorem} 

As in the $p\equiv_{(3)} 1$ case, the action of $\lm(Na_0)$ and $\lm(\delta)$ on $Z$ is essentially free.
Let $\tZ$ denote the span of the monomials of $Z$ which are reduced with respect to $\lm(Na_0)$ and $\lm(\delta)$.
Then 
$$HS(Z,t)=\frac{HS(\tZ,t)}{(1-t^{p+1})(1-t^{p(p-1)})}.$$

Define $\tZ_j$ to be the span of the monomials in $\tZ$ of the form $a_3^{pj}a_2^x a_1^y$.
Then $$\tZ=\bigoplus_{j=0}^{p-1}\tZ_j.$$

The calculation of $HS(\tZ_j,t)$ for $j<(p-1)/3$ is precisely as in the $p\equiv_{(3)} 1$ case.

For $j=\frac{p+1}{3}$ the elements of $\mcB$ are:

\begin{itemize}
\item $\alpha_{i,\frac{p+1}{3}}= a_3^{p\frac{p+1}{3}}a_1^{2+i(p-1)}$ for $i\in\bbN$, with parity $(0,0)$;
\item $\gamma_{\frac{p+1}{3}}=a_3^{p\frac{p+1}{3}}a_2^{2p-4}$ with parity $(0,0)$;
\item $\epsilon_{i,\frac{p+1}{3}}=a_3^{p\frac{p+1}{3}}a_2^pa_1^{3+i(p-1)}$ for $i\in\bbN$, with parity $(1,1)$;
\item $\mu=a_3^{p\frac{p+1}{3}}a_2^{2p-3}a_1$ with parity $(1,1)$.
\end{itemize}

For parity $(0,0)$: Observe that $\lm(D)\gamma_{\frac{p+1}{3}}=\lm(K)^2\alpha_{0,\frac{p+1}{3}}$
and $\alpha_{ij}\lm(K)=\alpha_{i-1,j}\lm(D)^{(p-1)/2}$ for $i>0$. Thus it is sufficient to count the
monomials $\alpha_{i+1,(p+1)/3}\lm(D)^{\ell}$, $\alpha_{0,(p+1)/3}\lm(D)^{\ell}\lm(K)^i$, and  
$\gamma_{(p+1)/3}\lm(K)^i$ for $i,\ell\in\bbN$.

For parity $(1,1)$: Observe that $\lm(D)\mu=\lm(K)\epsilon_{0,\frac{p+1}{3}}$
and $\epsilon_{ij}\lm(K)=\epsilon_{i-1,j}\lm(D)^{(p-1)/2}$ for $i>0$. Thus it is sufficient to count the
monomials  $\mu\lm(K)^i$ and $\epsilon_{i,(p+1)/3}\lm(D)^{\ell}$.

Counting monomials and identifying the appropriate geometric series gives
$$HS\left(\tZ_{\frac{p+1}{3}},t\right)=t^{p(p+1)/3}\left(\frac{t^2+t^{p+1}+t^{p+3}+t^{2p-2}}{(1-t^4)(1-t^{p-1})}
+\frac{t^{2p-4}}{1-t^{p-1}}\right).$$

\noindent We now consider the range $\frac{p+4}{3}\leq j\leq\frac{2p-4}{3}$. The following table indicates the monomials 
 and their respective parities:

\hbox{}

\begin{tabular}{|l|l|c|c|}
\hline \multicolumn{2}{|c|}{Monomial} & Parity $j$ even & Parity $j$ odd \\ \hline
$\alpha_{i,j}$ & $a_3^{pj}a_1^{3j-p+1+i(p-1)}$, $i\in\bbN$ & (0,0) & (0,1) \\ \hline
$\eta_j$ & $a_3^{pj}a_2^{\frac{5p-1-6j}{3}}a_1^{\frac{3j-p+5}{3}}$ & (0,0) & (0,1) \\ \hline
$\gamma_j$ & $a_3^{pj}a_2^{3p-3-3j}$ & (0,0) & (1,0) \\ \hline
$\Delta_j$ & $a_3^{pj}a_2^{3p-2-3j}a_1$ & (1,1) & (0,1) \\ \hline
$\epsilon_{i,j}$ & $a_3^{pj}a_2^pa_1^{3j-p+2+i(p-1)}$, $i\in\bbN$ & (1,1) & (1,0) \\ \hline
\end{tabular}

\hbox{} 

For $j$ even, parity $(0,0)$: We have $\eta_j\lm(K)=\gamma_j\lm(D)^{(3j-p+5)/6}$, 
$\alpha_{0j}\lm(K)=\eta_j\lm(D)^{j-(p+1)/3}$ and $\alpha_{ij}\lm(K)=\alpha_{i-1,j}\lm(D)^{(p-1)/2}$ for $i>0$.
Thus we need to count $\alpha_{ij}\lm(D)^{\ell}$, $\eta_j\lm(D)^{\ell}$ and $\gamma_j\lm(K)^i\lm(D)^{\ell}$.

For $j$ even, parity $(1,1)$: $\epsilon_{ij}\lm(K)=\epsilon_{i-1,j}\lm(D)^{(p-1)/2}$ and
$\epsilon_{0j}\lm(K)=\Delta_j\lm(D)^{(3j-p+1)/2}$. Thus we need to count
$\epsilon_{ij}\lm(D)^{\ell}$ and $\Delta_j\lm(K)^i\lm(D)^{\ell}$.

Counting monomials and identifying the appropriate geometric series gives:
\begin{equation*} 
HS(\tZ_{j},t)=t^{jp}\left(\frac{t^{3j+2}+t^{3p-3-3j}+t^{3p-1-3j}+t^{3j-p+1}}{(1-t^4)(1-t^{p-1})}
+\frac{t^{4(p+1)/3-j}}{1-t^4}\right).
\end{equation*}

For $j$ odd, the calculations are analogous with the roles of $\gamma_j$ and $\Delta_j$ reversed. The contribution 
to $HS(\tZ,t)$ is the same for both $j$ even and $j$ odd. Thus for $(p+1)/3<j<(2p-1)/3$ we have:

\begin{equation*} 
HS(\tZ_{j},t)=t^{jp}
\left(\frac{t^{3j+2}+t^{3p-3-3j}+t^{3p-1-3j}+t^{3j-p+1}+t^{4(p+1)/3-j}(1-t^{p-1})}{(1-t^4)(1-t^{p-1})}\right).
\end{equation*}

For $j=\frac{2p-1}{3}$ the monomials to consider are:

\begin{itemize}
\item $\alpha_{i,\frac{2p-1}{3}}= a_3^{p\frac{2p-1}{3}}a_1^{p+i(p-1)}$ for $i\in\bbN$, with parity $(0,1)$;
\item $\phi_{\frac{2p-1}{3}}=a_3^{p\frac{2p-1}{3}}a_2^{p-1}a_1$ with parity $(0,1)$;
\item $\eta_{\frac{2p-1}{3}}=a_3^{p\frac{2p-1}{3}}a_2^{\frac{p+1}{3}}a_1^{\frac{p+4}{3}}$ with parity $(0,1)$;
\item $\gamma_{\frac{2p-1}{3}}=a_3^{p\frac{2p-1}{3}}a_2^{2p-3}$ with parity $(1,0)$;
\item $\epsilon_{i,\frac{2p-1}{3}}=a_3^{p\frac{2p-1}{3}}a_2^pa_1^{p+1+i(p-1)}$ for $i\in\bbN$, with parity $(1,0)$;
\item $\lambda=a_3^{p\frac{p+1}{3}}a_2^pa_1^2$ with parity $(1,0)$.
\end{itemize}

For parity $(0,1)$: $\alpha_{ij}\lm(K)=\alpha_{i-1,j}\lm(D)^{(p-1)/2}$ for $i>0$, 
$\alpha_{0j}\lm(K)=\eta_j\lm(D)^{(p-2)/3}$ and $\eta_j\lm(K)=\phi_j\lm(D)^{(p+1)/6}$.
Thus we need to count  $\alpha_{ij}\lm(D)^{\ell}$, $\eta_j\lm(D)^{\ell}$ and $\phi_j\lm(K)^i\lm(D)^{\ell}$.

For parity $(1,0)$: $\epsilon_{ij}\lm(K)=\epsilon_{i-1,j}\lm(D)$ for $i>0$,
$\epsilon_{0j}\lm(K)=\lambda\lm(D)^{(p-1)/2}$ and $\lambda\lm(K)=\gamma_j\lm(D)$.
Thus we need to count  $\epsilon_{ij}\lm(D)^{\ell}$, $\lambda\lm(D)^{\ell}$ and $\gamma_j\lm(K)^i\lm(D)^{\ell}$.

Counting monomials and identifying the appropriate geometric series gives:
\begin{equation} 
HS\left(\tZ_{\frac{2p-1}{3}},t\right)=t^{p(2p-1)/3}
\left(\frac{2t^p+t^{2p-3}+t^{2p+1}}{(1-t^4)(1-t^{p-1})}
+\frac{t^{p+2}+t^{(2p+5)/3}}{1-t^4}\right).
\end{equation}

We now consider the range $\frac{2p+2}{3}\leq j\leq p-2$. The following table gives the relevant
 monomials and their parities:

\hbox{}

\begin{tabular}{|l|l|c|c|}
\hline \multicolumn{2}{|c|}{Monomial} & Parity $j$ even & Parity $j$ odd \\ \hline
$\alpha_{i,j}$ & $a_3^{pj}a_1^{3j-2p+2+i(p-1)}$, $i\in\bbN$ & (0,0) & (0,1) \\ \hline
$\phi_j$ & $a_3^{pj}a_2^{\frac{7p-5-6j}{3}}a_1^{\frac{3j-2p+4}{3}}$ & (0,0) & (0,1) \\ \hline
$\gamma_j$ & $a_3^{pj}a_2^{4p-4-3j}$ & (0,0) & (1,0) \\ \hline
$\Delta_j$ & $a_3^{pj}a_2^{4p-3-3j}a_1$ & (1,1) & (0,1) \\ \hline
$\epsilon_{i,j}$ & $a_3^{pj}a_2^pa_1^{3j-2p+3+i(p-1)}$, $i\in\bbN$ & (1,1) & (1,0) \\ \hline
\end{tabular}

\hbox{}

For $j$ even, parity $(0,0)$: We have $\phi_j\lm(K)=\gamma_j\lm(D)^{(3j-2p+4)/6}$, 
$\alpha_{0j}\lm(K)=\phi_j\lm(D)^{j-(2p-1)/3}$ and $\alpha_{ij}\lm(K)=\alpha_{i-1,j}\lm(D)^{(p-1)/2}$ for $i>0$.
Thus we need to count $\alpha_{ij}\lm(D)^{\ell}$, $\phi_j\lm(D)^{\ell}$ and $\gamma_j\lm(K)^i\lm(D)^{\ell}$.

For $j$ even, parity $(1,1)$: $\epsilon_{ij}\lm(K)=\epsilon_{i-1,j}\lm(D)^{(p-1)/2}$ and
$\epsilon_{0j}\lm(K)=\Delta_j\lm(D)^{(3j-2p+2)/2}$. Thus we need to count
$\epsilon_{ij}\lm(D)^{\ell}$ and $\Delta_j\lm(K)^i\lm(D)^{\ell}$.

Counting monomials and identifying the appropriate geometric series gives:
\begin{equation*} 
HS(\tZ_{j},t)=t^{jp}\left(\frac{t^{3j-2p+2}+t^{4p-4-3j}+t^{4p-4-3j}+t^{3j-p+3}}{(1-t^4)(1-t^{p-1})}
+\frac{t^{5(p+1)/3-j-2}}{1-t^4}\right).
\end{equation*}

For $j$ odd, the calculations are analogous with the roles of $\gamma_j$ and $\Delta_j$ reversed. The contribution 
to $HS(\tZ,t)$ is the same for both $j$ even and $j$ odd. Thus for $(2p-1)/3<j<p-1$ we have:

\begin{equation*} 
HS(\tZ_{j},t)=t^{jp}
\left(\frac{t^{3j+2-2p}+t^{4p-4-3j}+t^{4p-2-3j}+t^{3j-p+3}+t^{5(p+1)/3-j-2}(1-t^{p-1})}{(1-t^4)(1-t^{p-1})}\right).
\end{equation*}

Finally, we consider the case $j=p-1$. The only monomials we have here are

\begin{itemize}
\item $\alpha_{i,p-1}= a_3^{p(p-1)}a_1^{p-1+i(p-1)}$ for $i\in\bbN$, with parity $(0,0)$;
\item $\epsilon_{i,p-1}=a_3^{p(p-1)}a_2^pa_1^{p+i(p-1)}$ for $i\in\bbN$, with $(1,1)$.
\end{itemize}

Note that $\alpha_{0,p-1}\lm(K)=\lm(\delta)\lm(D)^{(p-1)/2}\not\in\tZ$ and, for $i>0$, we have
$\alpha_{i,p-1}\lm(K)=\alpha_{i-1,p-1}\lm(D)^{(p-1)/2}$. Similarly,
$$\epsilon_{0,p-1}\lm(K)=\lm(\delta)\lm(L)\lm(D)^{(p-1)/2}\not\in\tZ$$ and, for $i>0$,
$\epsilon_{i,p-1}\lm(K)=\epsilon_{i-1,p-1}\lm(D)^{(p-1)/2}$.
Thus it is sufficient
to count the monomials $\alpha_{i,p-1}\lm(D)^{\ell}$ and  $\epsilon_{i,p-1}\lm(D)^{\ell}$  with $i,\ell\in\bbN$.
Counting monomials and identifying the appropriate geometric series gives
$$HS(\tZ_{p-1},t)=\frac{t^{p(p-1)}(t^{p-1}+t^{2p})}{(1-t^4)(1-t^{p-1})}.$$

Summing over $j$ and simplifying gives

\begin{equation*} HS(Z,t)=\frac{Numer(t)}{Denom(t)} \end{equation*}

\noindent where
\begin{eqnarray*} Numer(t)&=&\chi_1(t)(1-t^{p-3})(1-t^{p+3})+\chi_2(t)(1-t^{p+3})+\chi_3(t)(1-t^{p-3}),\\
\chi_1(t)&=& 1+t^{p+1}+t^{p(p+1)}+t^{(p+1)(p-1)} +t^p(t^3+t^{p-2}+t^{p+4}+t^{2p-5})\\
 &+&t^{p(p+1)/3}(t^2+t^{p+1}+t^{p+3}+t^{2p-4}+t^{2p-2}-t^{2p})\\
&+&t^{p(2p-1)/3}(2t^p+t^{p+2}+t^{2p-3}+t^{(2p+5)/3}(1-t^{p-1}))\\
&+&t^{3(p-1)}(1-t^{(p-1)(p-5)/3})(1+t^{p(p-2)/3+3}+t^{2p(p-2)/3+2}),\\
\chi_2(t)&=&t^{4(p-2)}(1-t^{(p-3)(p-5)/3})(1+t^2)(1+t^{p(p-2)/3+1}+t^{2p(p-2)/3+2}) ,\\
\chi_3(t)&=&t^{2p+6}(1-t^{(p+3)(p-5)/3})(1+t^{p+1})(1+t^{p(p-2)/3-1}+t^{2p(p-2)/3-2})\\
& {\rm and}&\\
 Denom(t)&=&(1-t^4)(1-t^{p-3})(1-t^{p-1})(1-t^{p+1})(1-t^{p+3})(1-t^{p(p-1)}).
\end{eqnarray*}

\noindent This agrees with the calculation of $HS(\field[V]^G,t)$ by Hughes-Kemper \cite[2.7(d)]{HK}.

\section{Concluding Remarks}

We do not claim that the generating sets given in Theorem~\ref{maintheorem} are minimal. 
However, for $p=5$ and $p=7$, MAGMA~\cite{magma} calculations confirm that the given sets are
minimal generating sets. Recall that the {\it Noether number} is the maximum degree of an element
in a minimal homogeneous generating set. Thus the Noether number is $22$ for $p=5$ and $16$ for $p=7$.
Examining the degrees of the polynomials occurring in Theorem~\ref{maintheorem} gives the following.

\begin{Cor} The Noether number of $\mathbb{F}[V]^{SL_2(\mathbb{F}_p)}$ is bounded above by 
\begin{itemize}
\item $p^2-p+4$ if $p \equiv_{(3)} -1$,
\item $\frac{p^2-p+12}{3}$ if $p \equiv_{(3)} 1$.
\end{itemize} 
\label{NNcor} \end{Cor} 

It follows from the proof of Theorem~\ref{maintheorem} that $\mcG$ is a SAGBI basis for $\mathbb{F}[V]^{SL_2(\mathbb{F}_p)}$.
This means that the set $\lm(\mcG)$ generates the lead term algebra of $\mathbb{F}[V]^{SL_2(\mathbb{F}_p)}$ and
if $f\in\mathbb{F}[V]^{SL_2(\mathbb{F}_p)}$ then $\lm(f)$ can be written as a product of elements from $\lm(\mcG)$.

\begin{Cor}  $\mathbb{F}[V]^{SL_2(\mathbb{F}_p)}$ does not have a finite SAGBI basis using the graded reverse lexicographical order
with $a_0<a_1<a_2<a_3$.
\label{nofintesagbicor} \end{Cor}
\begin{proof}
Observe that if $a_1^j \in LM(\mcG)$ then $j=0$ and if $m \in LM(\mcG)$ with $a_3$ dividing $m$, then $a_3^p$ divides $m$. 
Thus $LM(h_i)=a_3^pa_1^{p+2+(i-1)(p-1)}$ is indecomposable in the lead term algebra of $\mathbb{F}[V]^{SL_2(\mathbb{F}_p)}$.
\end{proof}

\ifx\undefined\bysame
\newcommand{\bysame}{\leavevmode\hbox to3em{\hrulefill}\,}
\fi

\end{document}